\documentclass{amsart}

\input xy

\xyoption{all}

\usepackage{amssymb,amsbsy,amsmath,mathrsfs,graphicx}
\usepackage{times}

\newtheorem{thm}{Theorem}
\newtheorem{lem}[thm]{Lemma}

\newtheorem{cor}[thm]{Corollary}

\newtheorem*{rmk}{Remark}

\newcommand{\bs}[1]{\boldsymbol{#1}}

\newcommand{\bbC}{\mathbb{C}}
\newcommand{\bbD}{\mathbb{D}}

\newcommand{\bbN}{\mathbb{N}}

\newcommand{\sfE}{\mathsf{E}}

\renewcommand{\d}{\mathrm{d}}

\renewcommand{\S}{\Sigma}

\renewcommand{\a}{\alpha}

\newcommand{\s}{\sigma}
\renewcommand{\phi}{\varphi}

\newcommand{\Hom}{\mathrm{Hom}}

\newcommand{\ol}[1]{\overline{#1}}

\newcommand{\fin}{\nolinebreak\hspace{\stretch{1}}$\lhd$}

\renewcommand{\t}[1]{\widetilde{#1}}
\renewcommand{\to}{\longrightarrow}

\newcommand{\actson}{\curvearrowright}

\begin{document}

\title[Quantitative equidistribution in quasi-random groups]{Quantitative equidistribution for certain quadruples in quasi-random groups}


\author{Tim Austin}
\address{Courant Institute, New York University, New York, NY 10012, U.S.A.}
\email{tim@cims.nyu.edu}
\thanks{Research supported by a fellowship from the Clay Mathematics Institute}

\date{}

\maketitle

\begin{abstract}
Bergelson and Tao have recently proved that if $G$ is a $D$-quasi-random group, and $x,g$ are drawn uniformly and independently from $G$, then the quadruple $(g,x,gx,xg)$ is roughly equidistributed in the subset of $G^4$ defined by the constraint that the last two coordinates lie in the same conjugacy class. Their proof gives only a qualitative version of this result. The present notes gives a rather more elementary proof which improves this to an explicit polynomial bound in $D^{-1}$.
\end{abstract}

\vspace{7pt}


\vspace{7pt}

Let $G$ be a compact group, and let $(X,\S,\mu)$ be $G$ with its Borel $\s$-algebra and Haar probability measure, regarded as an abstract probability space.  The regular representations give two natural $\mu$-preserving $G$-actions on $X$:
\[S^gx := gx \quad \hbox{and} \quad T^gx := xg^{-1}.\]
Observe that $S$ and $T$ commute, and that $S^gT^g$ is the action of $g$ on $X = G$ by conjugation.  Let $\Phi \leq \S$ be the sub-$\s$-algebra of conjugation-invariant (that is, $ST$-invariant) sets.  Also, let $\bbD := \{z \in \bbC:\ |z|\leq 1\}$.  

Following Gowers~\cite{Gow08}, the group $G$ is \emph{$D$-quasi-random} if it has no non-trivial representations of dimension less than $D$.  This note will prove the following.

\begin{thm}\label{thm:BerTao}
If $G$ is $D$-quasi-random, then
\begin{multline}\label{eq:BerTao}
\int_G\Big|\int_X f_1(x)f_2(gx)f_3(xg)\,\mu(\d x)\\ - \Big(\int_X f_1\,\d\mu\Big)\Big(\int_X\sfE(f_2\,|\,\Phi)\sfE(f_3\,|\,\Phi)\,\d\mu\Big)\Big|\,\d g \leq 4D^{-1/8}
\end{multline}
for all measurable functions $f_1,f_2,f_3:X\to \bbD$.
\end{thm}

In the recent paper~\cite{BerTao14}, Bergelson and Tao prove that there is some upper bound $c(D)$ for the left-hand side of~(\ref{eq:BerTao}) which tends to $0$ as $D\to \infty$.  (They consider only finite groups $G$, but this is inconsequential.) Their method does not give an effective formula for $c(D)$, but they conjecture that it is polynomial in $D^{-1}$, so Theorem~\ref{thm:BerTao} confirms this.  The proof of Theorem~\ref{thm:BerTao} below also seems much more direct than theirs.  Their approach uses a passage to an ultralimit along a sequence of increasingly quasi-random groups, followed by results from~\cite{BerMcC07} concerning limits along idempotent ultrafilters in infinite groups.  This is why their estimate is ineffective.  The proof below has a few early steps in common with theirs, but then uses only an elementary inequality from representation theory.  In this, it is rather closer to Gowers' original estimates for quasi-random groups in~\cite[Section 4]{Gow08}.  Its format is also similar to Furstenberg's proof in~\cite{Fur77} that weakly mixing transformations are weakly mixing for multiple recurrence.

After writing the first version of this paper, I learned of another short proof of the Bergelson--Tao result due to Anush Tserunyan.  Her initial version again used ultraproducts, but some small modifications gave another effective proof, which improves the bound in Theorem~\ref{thm:BerTao} to $4D^{-1/4}$.  That argument is presented in~\cite[Section 5]{Tse14}.

The reader may consult~\cite{BerTao14} for a discussion of the interpretation of Theorem~\ref{thm:BerTao} in terms of the distribution of the quadruple $(g,x,gx,xg)$ when $x,g$ are drawn uniformly and independently at random from $\mu$.  That paper also derives some combinatorial consequences of Theorem~\ref{thm:BerTao}, discusses possible generalizations to larger values of $d$, and gives some related results that can be obtained by more straightforward combinatorial arguments.

The basis of our argument is the following inequality.  In its formulation, if $\pi:G\actson V$ is a unitary representation, then we let $V^\circ_\pi$ denote its trivial component (that is, the subspace of $\pi$-fixed points), and let $P^\circ_\pi:V\to V^\circ_\pi$ be the orthogonal projection.

\begin{lem}\label{lem:Schur}
If $G$ is $D$-quasi-random, $\pi:G\actson V$ is a unitary representation, $\|\cdot\|_V$ denotes the norm on $V$, and $u,v \in V$, then 
\[\big\|P^\circ_{\pi\otimes \pi}(u\otimes v) - P^\circ_\pi u\otimes P^\circ_\pi v\big\|_{V_{\pi\otimes \pi}} \leq D^{-1/2}\|u\|_V\|v\|_V.\]
\end{lem}

\begin{proof}
Let $\pi = \bigoplus_{i\geq 0}\rho_i$ be a decomposition of $\pi$ into (not necessarily distinct) irreducibles, let $V_i \leq V$ be the direct summand corresponding to $\rho_i$, let $d_i := \dim V_i$, and let $u = \bigoplus_iu_i$ and $v = \bigoplus_iv_i$ be the corresponding vector decompositions. This decomposition of $\pi$ gives
\[V^\circ_{\pi\otimes \pi} = \bigoplus_{i,j}V^\circ_{\rho_i\otimes \rho_j} \leq \bigoplus_{i,j}V_i\otimes V_j \leq \Big(\bigoplus_iV_i\Big)\otimes \Big(\bigoplus_j V_j\Big).\]
As is standard, for each $i,j$ one may identify $V_i\otimes V_j$ with the space $\mathrm{Hom}(V_j^\ast,V_i)$, which becomes a Hilbert space when endowed with the trace inner product,
\[\langle S,T\rangle_{\mathrm{Hom}(V_j^\ast,V_i)} := \mathrm{tr}\, S^\ast T = \mathrm{tr}\, T^\ast S,\]
and is given the action
\[(\rho_i\otimes \rho_j)^gT := \rho^g_i\circ T\circ (\rho_j^g)^\ast \quad \hbox{for}\ g \in G,\ T \in \mathrm{Hom}(V_j^\ast,V_i).\]

By Schur's Lemma (see, for instance,~\cite[Chapter 2]{Bum04}), one has $V^\circ_{\rho_i\otimes \rho_j} = 0$ unless $\rho_j \cong \rho^\ast_i$.  In that case $\mathrm{Hom}(V_j^\ast,V_i) \cong \mathrm{End}(\bbC^{d_i})$, and under this isomorphism, $V^\circ_{\rho_i\otimes \rho_i^\ast}$ is identified with the one-dimensional subspace generated by the identity matrix $\bs{1}_{d_i} \in \mathrm{End}(\bbC^{d_i})$, which has trace-norm equal to $d_i^{1/2}$.

Let $\mathrm{R}$ be the relation on $\bbN$ defined by
\[i \mathrm{R} j \quad \Longleftrightarrow \quad \rho_j = \rho_i^\ast,\]
and for pairs $i,j$ such that $i\mathrm{R}j$, let $I_{ij}$ be the element of $\Hom(V_j^\ast,V_i)$ that corresponds to $d_i^{-1/2}\bs{1}_{d_i}$. Substituting the above consequence of Schur's Lemma into the expression for $V^\circ_{\pi\otimes \pi}$, we obtain
\[V^\circ_{\pi\otimes \pi} = \bigoplus_{i\mathrm{R}j}\bbC\cdot I_{ij}.\]
This now gives
\[P^\circ_{\pi\otimes \pi}(u\otimes v) = \bigoplus_{i\mathrm{R}j}\langle u_i\otimes v_j,I_{ij}\rangle_{V_i\otimes V_j}\cdot I_{ij} = \bigoplus_{i \mathrm{R} j}d_i^{-1/2}\langle u_i, v_j\rangle_{V_i} \cdot I_{ij}.\]
On the other hand, if $I := \{i\in \bbN:\ \rho_i = \mathrm{triv}\}$, then
\[P^\circ_\pi u\otimes P^\circ_\pi v = \bigoplus_{i,j\in I}u_i\otimes v_j.\]
Subtracting the latter from the former and computing norms, we obtain
\begin{multline*}
\big\|P^\circ_{\pi\otimes \pi}(u\otimes v) - P^\circ_\pi u\otimes P^\circ_\pi v\big\|_{V_{\pi\otimes \pi}}^2
= \sum_{\scriptsize{\begin{array}{c}i,j \in \bbN\setminus I,\\ i \mathrm{R} j\end{array}}}d_i^{-1}|\langle u_i,v_j\rangle_{V_i}|^2\\
\leq D^{-1}\sum_{i \mathrm{R} j}|\langle u_i,v_j\rangle_{V_i}|^2
\leq D^{-1}\Big(\sum_i\|u_i\|_{V_i}^2\Big)\Big(\sum_i \|v_i\|_{V_i}^2\Big)\leq D^{-1}\|u\|_V^2\|v\|_V^2,
\end{multline*}
as required.
\end{proof}

\begin{cor}\label{cor:Schur}
With the same data as above, one has
\[\int_G|\langle u,\pi^g v\rangle_V - \langle P^\circ_\pi u,P^\circ_\pi v\rangle_V|^2\,\d g \leq D^{-1/2}\|u\|_V^2\|v\|_V^2.\]
\end{cor}

\begin{rmk}
In the published version of this paper, the proof given for Corollary~\ref{cor:Schur} was incorrect.  A correct proof has been published as an erratum in the same journal.  This preprint has been re-written with the correct proof.  It actually gives an improved bound in which $D^{-1/2}$ becomes $D^{-1}$, but I have not changed the rest of the preprint to account for this. \fin
\end{rmk}

\begin{proof}
Replacing $u$ with $u - P_\pi^\circ u$ and $v$ with $v - P^\circ_\pi v$, we may assume that $\pi$ has no trivial component.  Let $\pi = \bigoplus_{i\geq 1}\rho_i$, $u = \bigoplus_{i\geq 1}u_i$ and $v = \bigoplus_{i\geq 1}v_i$ be a decomposition into irreducible components as in the previous proof.  Substituting into the desired integral gives
\[\int_G |\langle u,\pi^g v\rangle_V|^2\,\d g = \int_G \langle u,\pi^g v\rangle\ol{\langle u,\pi^g v\rangle}\,\d g = \sum_{i,j \geq 1}\int_G \langle u_i,\rho_i^g v_i\rangle_{V_i}\ol{\langle u_j,\rho_j^g v_j\rangle_{V_j}}\,\d g.\]
Another appeal to Schur Orthogonality (\cite[Theorems 2.3 and 2.4]{Bum04}) evaluates each term in this sum.  The result is
\[\sum_{\scriptsize{\begin{array}{c}i,j \geq 1\\ \rho_i \cong \rho_j\end{array}}}\frac{1}{\dim(V_i)}\langle u_i,u_j\rangle_{V_i}\langle v_i,v_j\rangle_{V_i}.\]
Here we have used that when $\rho_i \cong \rho_j$, the isomorphism is unique, and so we may canonically identify $u_j$ and $v_j$ with elements of $V_i$ to compute the inner products.

Finally, since $G$ is $D$-quasi-random and $\pi$ contains no trivial subrepresentation, the above sum is bounded by
\begin{multline*}
D^{-1}\sum_{\scriptsize{\begin{array}{c}i,j \geq 1\\ \rho_i \cong \rho_j\end{array}}}\|u_i\|_{V_i}\cdot \|u_j\|_{V_j}\cdot \|v_i\|_{V_i}\cdot \|v_j\|_{V_j} \leq D^{-1}\Big(\sum_{i\geq 1}\|u_i\|_{V_i}\|v_i\|_{V_i}\Big)^2\\
\leq D^{-1}\Big(\sum_{i\geq 1}\|u_i\|_{V_i}^2\Big)\Big(\sum_{i\geq 1}\|v_i\|_{V_i}^2\Big) = D^{-1}\|u\|_V^2\|v\|_V^2,
\end{multline*}
where the second estimate is by the Cauchy--Bunyakowski--Schwartz Inequality.
\end{proof}

\begin{proof}[Proof of Theorem~\ref{thm:BerTao}]\emph{Step 1: Initial re-arrangement.}\quad 
We start in the same way as~\cite{BerTao14}.  If $f_1$ is constant, say equal to $\a$, then $|\a| \leq 1$, and an easy calculation gives
\begin{eqnarray*}
&&\int_G\Big|\int_X f_1(x)f_2(gx)f_3(xg)\,\mu(\d x) - \Big(\int_X f_1\,\d\mu\Big)\Big(\int_X\sfE(f_2\,|\,\Phi)\sfE(f_3\,|\,\Phi)\,\d\mu\Big)\Big|\,\d g\\
&&= |\a|\int_G\Big|\int_X f_2(gx)f_3(xg)\,\mu(\d x) - \int_X\sfE(f_2\,|\,\Phi)\sfE(f_3\,|\,\Phi)\,\d\mu\Big|\,\d g\\
&& \leq \Big(\int_G\Big|\int_X f_2(gxg^{-1})f_3(x)\,\mu(\d x) - \int_X\sfE(f_2\,|\,\Phi)\sfE(f_3\,|\,\Phi)\,\d\mu\Big|^2\,\d g\Big)^{-1/2},
\end{eqnarray*}
where the last estimate uses the Cauchy--Bunyakowski--Schwartz Inequality.  Let $V:= L^2(G)$, $\pi:G\actson V$ be the action of composition by conjugation on $G$, and let $u := f_3$ and $v := \ol{f_2}$.  Then these both have $\|\cdot\|_2$-norm at most $1$, and so Corollary~\ref{cor:Schur} bounds the last line above by $D^{-1/4} \leq D^{-1/8}$.

By multi-linearity and a change of variables in the inner integral, it therefore suffices to show that
\begin{eqnarray*}
\int_G\Big|\int_X f_1(x)f_2(gx)f_3(xg)\,\mu(\d x)\Big|\,\d g
&=& \int_G\Big|\int_X f_3(y)f_1(yg^{-1})f_2(gyg^{-1})\,\mu(\d y)\Big|\,\d g\\ &=& \int_G\Big|\int_X f_3\cdot f_1T^g\cdot f_2S^gT^g\,\d\mu\Big|\,\d g \leq 3D^{-1/8},
\end{eqnarray*}
whenever $f_2,f_3:X\to \bbD$, $f_1:X\to 2\bbD$ with $\|f_1\|_2 \leq 1$ and $\int f_1\,\d\mu = 0$.

\vspace{7pt}

\emph{Step 2: Removing the absolute values.} \quad By the Cauchy--Bunyakowski--Schwartz Inequality, the previous inequality will follow if one shows that
\begin{multline*}
\int_G\Big|\int_X f_3\cdot f_1T^g\cdot f_2S^gT^g\,\d\mu\Big|^2\,\d g = \int_G\int_{X^2} F_3\cdot F_1\t{T}^g\cdot F_2\t{S}^g\t{T}^g\,\d\mu^{\otimes 2}\,\d g\\
= \int_{X^2} F_3\cdot \Big(\int_G F_1\t{T}^g\cdot F_2\t{S}^g\t{T}^g\,\d g\Big)\,\d\mu^{\otimes 2} \leq 5 D^{-1/4},
\end{multline*}
where $F_i:= f_i\otimes \ol{f_i}:X^2\to \bbC$ for $i=1,2,3$, and $\t{S} := S\times S$, $\t{T} := T\times T$.  The key here is that we have removed the absolute values inside the outer integral, which makes it possible to change the order of the integrals in the last step above.

\vspace{7pt}

\emph{Step 3: Expansion and another change of variables.}\quad Since $\|F_3\|_2 \leq 1$, another appeal to the Cauchy--Bunyakowski--Schwartz Inequality shows that the above would follow from
\begin{eqnarray*}
&&\int_{X^2} \Big|\int_G F_1\t{T}^g\cdot F_2\t{S}^g\t{T}^g\,\d g\Big|^2\,\d\mu^{\otimes 2}\\
&&= \int_{X^2} \int_G \int_G F_1\t{T}^g\cdot \ol{F_1}\t{T}^{hg}\cdot F_2\t{S}^g\t{T}^g\cdot \ol{F_2}\t{S}^{hg}\t{T}^{hg}\,\d g\,\d h\,\d\mu^{\otimes 2}\\
&&= \int_G \int_G \int_{X^2} F_1\t{T}^g\cdot \ol{F_1}\t{T}^{hg}\cdot F_2\t{S}^g\t{T}^g\cdot \ol{F_2}\t{S}^{hg}\t{T}^{hg}\,\d\mu^{\otimes 2}\,\d g\,\d h \leq 25D^{-1/2}.
\end{eqnarray*}

For the proof of this last inequality, we may change variables in the inner integral by $\t{T}^{g^{-1}}$, and so prove instead that
\begin{multline}\label{eq:BerTaointermed}
\int_G \int_G \int_{X^2} F_1\cdot \ol{F_1}\t{T}^h\cdot (F_2\cdot \ol{F_2}\t{S}^h\t{T}^h)\t{S}^g\,\d\mu^{\otimes 2}\,\d g\,\d h\\
= \int_G \int_{X^2} F_1\cdot \ol{F_1}\t{T}^h\cdot \sfE(F_2\cdot \ol{F_2}\t{S}^h\t{T}^h\,|\,\Delta)\,\d\mu^{\otimes 2}\,\d h \leq 25D^{-1/2},
\end{multline}
where $\Delta \leq \S^{\otimes  2}$ is the $\s$-algebra of $\t{S}$-invariant sets.

\vspace{7pt}

\emph{Step 4: Appeal to quasi-randomness.}\quad Now recall that
\[F_2\cdot \ol{F_2}\t{S}^h\t{T}^h = (f_2\cdot \ol{f_2}S^hT^h)\otimes (\ol{f_2}\cdot f_2S^hT^h),\]
a tensor product of two functions $X\to\bbD$.  We may therefore apply Lemma~\ref{lem:Schur} to obtain
\[\Big\|\sfE(F_2\cdot \ol{F_2}\t{S}^h\t{T}^h\,|\,\Delta) - \Big|\int_X f_2\cdot \ol{f_2}S^hT^h\,\d\mu\Big|^2\Big\|_{L^2(\mu^{\otimes 2})} \leq D^{-1/2} \quad \forall h \in G.\]

Substituting this into~(\ref{eq:BerTaointermed}), and using that
\[\|F_1\cdot \ol{F_1}\t{T}^h\|_{L^2(\mu^{\otimes 2})} = \|f_1\cdot \ol{f_1}T^h\|_{L^2(\mu)}^2 \leq (2\cdot 2)^2 = 16,\]
we see that~(\ref{eq:BerTaointermed}) will be proved (with room to spare) if one shows that
\begin{multline*}
\int_G \Big(\int_{X^2} F_1\cdot \ol{F_1}\t{T}^h\,\d\mu^{\otimes 2}\Big)\cdot \Big|\int_X f_2\cdot \ol{f_2}S^hT^h\,\d\mu\Big|^2\,\d h\\
= \int_G \Big|\int_X f_1\cdot \ol{f_1}T^h\,\d\mu\Big|^2\cdot \Big|\int_X f_2\cdot \ol{f_2}S^hT^h\,\d\mu\Big|^2\,\d h \leq D^{-1/2}.
\end{multline*}

Finally, this last inequality holds because
\[\int_G \Big|\int_X f_1\cdot \ol{f_1}T^h\,\d\mu\Big|^2\cdot \Big|\int_X f_2\cdot \ol{f_2}S^hT^h\,\d\mu\Big|^2\,\d h \leq \|f_2\|_\infty^4\int_G \Big|\int_X f_1\cdot \ol{f_1}T^h\,\d\mu\Big|^2\,\d h,\]
and by Corollary~\ref{cor:Schur} this is bounded by
\[D^{-1/2}\|f_1\|_2^4\|f_2\|_\infty^4 \leq  D^{-1/2},\]
since $\int f_1\,\d\mu = 0$.
\end{proof}

\bibliographystyle{abbrv}
\bibliography{bibfile}

\end{document}